\documentclass[11pt]{article}
\date{}

\usepackage[title]{appendix}
\usepackage{xcolor}
\usepackage[margin=1in]{geometry}                
\usepackage{graphicx}
\usepackage{subcaption}
\usepackage{pdflscape}
\usepackage{amssymb}
\usepackage[normalem]{ulem}
\usepackage{hyperref}
\usepackage{enumitem}
\usepackage{mathrsfs}
\usepackage{rotating}
\usepackage{longtable} 
\usepackage{adjustbox}
\usepackage{float}

\usepackage{color}
\usepackage{bbm, dsfont}
\usepackage{pst-node}
\usepackage{tikz-cd}
\usepackage{amsfonts} 
\usepackage{geometry}
\usepackage{amsthm}
\usepackage{amsmath}
\usepackage{titlesec}
\usepackage{amssymb}
\usepackage{enumitem}
\usepackage{float}
\usepackage [english]{babel}
\usepackage [autostyle, english = american]{csquotes}
\usepackage{algorithm}
\usepackage[noend]{algpseudocode} 
\makeatletter
\def\BState{\State\hskip-\ALG@thistlm}
\makeatother
\usepackage{hyperref}
\usepackage[normalem]{ulem}

\newlist{casess}{enumerate}{1}
\setlist[casess]{label=     \textbf{Case} \arabic*:}
\usepackage{mathtools}

\DeclarePairedDelimiter\floor{\lfloor}{\rfloor}

\makeatletter
\newcommand*{\rom}[1]{\expandafter\@slowromancap\romannumeral #1@}
\makeatother

\usepackage{etoolbox}

\makeatletter
\patchcmd{\ttlh@hang}{\parindent\z@}{\parindent\z@\leavevmode}{}{}
\patchcmd{\ttlh@hang}{\noindent}{}{}{}
\makeatother

\usepackage{listings}
\usepackage{color} 
\definecolor{mygreen}{RGB}{28,172,0} 
\definecolor{mylilas}{RGB}{170,55,241}

\newlist{Assumptions}{enumerate}{1}
\setlist[Assumptions]{label=     \textbf{Assumption} \arabic*:}

\makeatletter

\newsavebox{\@brx}
\newcommand{\llangle}[1][]{\savebox{\@brx}{\(\m@th{#1\langle}\)}%
  \mathopen{\copy\@brx\kern-0.5\wd\@brx\usebox{\@brx}}}
\newcommand{\rrangle}[1][]{\savebox{\@brx}{\(\m@th{#1\rangle}\)}%
  \mathclose{\copy\@brx\kern-0.5\wd\@brx\usebox{\@brx}}}
\makeatother

\usepackage{lipsum} 
\usepackage{titlesec}
\titleformat{\subsection}[runin]
       {\normalfont\bfseries}
       {\thesubsection}
       {0.5em}
       {}
       [.]

\newcommand{\A}{\mathfrak{A}}

\newcommand{\B}{\mathfrak{B}} 

\newcommand{\CC}{\mathbb{C}}

\usepackage{tikz-cd}
\usepackage{graphicx}

\def\N{\mathbb{N}}
\def\T{\mathbb{T}}
\def\Z{\mathbb{Z}}
\def\E{{\mathscr E}}

\def\Z{\mathbb Z}

\def\e{{\sf e}}

\def\BofH{\mathbb B(\mathcal H)}

\def\bu{\bullet}

\def\({\left(}
\def\[{\left[}
\def\){\right)}
\def\]{\right]}

\def\<{\langle}
\def\>{\rangle}
\providecommand{\norm}[1]{\lVert#1\rVert}

 \newtheorem{thm}{Theorem}[section]
 \newtheorem{cor}[thm]{Corollary}
 
 \newtheorem{lem}[thm]{Lemma}
 \newtheorem{prop}[thm]{Proposition}
 \theoremstyle{definition}
 \newtheorem{defn}[thm]{Definition}
 \theoremstyle{remark}
 
 \newtheorem{ex}[thm]{Example}
 \newtheorem{que}[thm]{Question}
 \numberwithin{equation}{section}

\numberwithin{equation}{section}

\usepackage[backend=biber,maxnames=10]{biblatex}
\usepackage{biblatex2bibitem}

\begin{document}


\title{On the topological ranks of Banach $^*$-algebras associated with groups of subexponential growth}

\author{Felipe I. Flores
\footnote{
\textbf{2020 Mathematics Subject Classification:} Primary 46K05, Secondary 43A20, 43A15.
\newline
\textbf{Key Words:} topological stable rank, real rank, Banach $^*$-algebra, symmetrized $L^p$-crossed product, subexponential growth, Fell bundle.
}
}

\maketitle


\begin{abstract}
Let $G$ be a group of subexponential growth and $\mathscr C\overset{q}{\to}G$ a Fell bundle. We show that any Banach $^*$-algebra that sits between the associated $\ell^1$-algebra $\ell^1( G\,\vert\,\mathscr C)$ and its $C^*$-envelope has the same topological stable rank and real rank as $\ell^1( G\,\vert\,\mathscr C)$. We apply this result to compute the topological stable rank and real rank of various classes of symmetrized twisted $L^p$-crossed products and show that some twisted $L^p$-crossed products have topological stable rank 1. Our results are new even in the case of (untwisted) group algebras.
\end{abstract}

\tableofcontents

\section{Introduction}\label{introduction}

The (left) topological stable rank of a unital Banach algebra $\B$ was introduced by Rieffel \cite{Ri83} as the smallest $n\in\Z^+$ such that $$
{\rm Lg}_n(\B) = \left\{ (b_1,\ldots,b_n)\in \B^n \mid \B=\B b_1+\ldots +\B b_n\right\}
$$ 
is dense in $\B^n$. The idea was to produce a noncommutative notion of dimension similar to both the covering dimension for compact spaces and the Bass stable rank. This idea was then applied very successfully to the study of K-theory for $C^*$-algebras \cite{Ri83b} and motivated the notion of real rank. Indeed, the real rank of a unital $C^*$-algebra $\B$, as introduced by Brown and Pedersen \cite{BP91}, is the smallest $n\in\N$ such that 
$$
{\rm Lg}_{n+1}(\B)_{\rm sa} = \left\{ (b_0,\ldots,b_n)\in \B_{\rm sa}^{n+1} \mid \B=\B b_0+\ldots +\B b_n\right\}
$$ 
is dense in $\B_{\rm sa}^{n+1}$.

Studying these noncommutative notions of dimension has become a very interesting problem for $C^*$-algebraists and has motivated many works in this area. This is particularly true in the case of group $C^*$-algebras and crossed products. Just to mention some examples, the topological stable rank has been studied for group $C^*$-algebras in \cite{SuTa95,DyHaRo97,GO20,Ra24}, whereas the real rank of group $C^*$-algebras has been studied in \cite{Ka93,ArKa12,Sc19,Mo23}. In the setting of crossed products, we can also find the works \cite{Pu90b,Pu90,CP05}.

There are also many works on topological ranks that remain outside the realm of $C^*$-algebras (see \cite{Ba98,Ba99,DaJi09}), and this will be our case as well. Indeed, we will be interested in the topological ranks of generalized convolution algebras and $L^p$-operator algebras, as introduced by N. C. Phillips \cite{CP12,CP13}.

The study of $L^p$-operator algebras has attracted much attention in the last decade (see \cite{Ga21} and the references therein), and it is currently a very active area of research \cite{BiDeWe24,ChGaTh24,El25}. Most of the articles seem to be concerned with the properties that work exactly as in the $C^*$-case \cite{El25}, while a few others deal with new (perhaps surprising) phenomena \cite{BiDeWe24,ChGaTh24}. In this sense, the present article should be placed in the first group, as we plan to demonstrate that the topological stable rank is independent of $p$ for a large class of symmetrized crossed products. Symmetrized $L^p$-crossed products are `$^*$-analogs' of regular $L^p$-crossed products and have the advantage of being Banach $^*$-algebras that naturally act on a $L^2$-space. They have been studied in \cite{CP19,SaWi20,AuOr22,el24b,El25,BaKw25}, often under the name of `algebras of pseudofunctions'.

In the case of topological stable rank 1, we can also say something about regular $L^p$-crossed products. In fact, we also aim to solve the following question of N. C. Phillips \cite[Question 8.4]{CP13}.
\begin{que}\label{theq}
    Let $\alpha$ be a minimal action of $\mathbb Z$ either on the Cantor set $X=2^\mathbb N$ or on $X=\mathbb T$ by irrational rotations. Does $F^p(\mathbb Z,X,\alpha)$ have topological stable rank 1?
\end{que}

Instead of $C^*$-dynamical systems, we chose to work with Fell bundles over groups of subexponential growth. This is in order to maximize generality, but also because the main tools we used were already developed in this framework \cite{BeCo25,Fl24}. Note that the inclusion of Fell bundles provides almost no extra complication (when the approach is compared to, say, an approach based strictly on twisted $C^*$-dynamical systems). We also hope that this generality will be useful in the future.

Now let us present the main theorem. Given a Fell bundle $\mathscr C\overset{q}{\to}G$, we let $\ell^1( G\,\vert\,{\mathscr C})$ be the associated algebra of summable cross-sections, and we write ${\rm C^*}( G\,\vert\,{\mathscr C})$ for the associated enveloping $C^*$-algebra. Naturally, $\ell^1( G\,\vert\,{\mathscr C})$ is equipped with its $1$-norm, whereas ${\rm C^*}( G\,\vert\,{\mathscr C})$ is considered with its $C^*$-norm. In what follows, we use the symbols ${\rm tsr}(\B)$ and ${\rm rr}(\B)$ to denote the topological stable rank and the real rank of the algebra $\B$.

\begin{thm}\label{mainthm}
    Let $G$ be a group of subexponential growth and $\mathscr C\overset{q}{\to}G$ a Fell bundle. Suppose that $\B$ is a Banach $^*$-algebra with a continuous involution and such that $\ell^1( G\,\vert\,{\mathscr C})\subset\B\subset {\rm C^*}( G\,\vert\,{\mathscr C})$. Then 
    $${\rm tsr}(\B)= {\rm tsr}\big(\ell^1( G\,\vert\,{\mathscr C})\big)\quad \text{ and }\quad {\rm rr}(\B)= {\rm rr}\big(\ell^1( G\,\vert\,{\mathscr C})\big).
    $$
\end{thm}

\begin{cor}[Corollary \ref{easycor}]
    Let $G$ be a group of subexponential growth and $\mathscr C\overset{q}{\to}G$ a Fell bundle. Let $\ell^1( G\,\vert\,{\mathscr C})\subset\B$ be any continuous, dense inclusion into a Banach algebra $\B$. Then 
    $$
    {\rm tsr}(\B)\leq {\rm tsr}\big({\rm C^*}( G\,\vert\,{\mathscr C})\big).
    $$
    In particular, if ${\rm tsr}\big({\rm C^*}( G\,\vert\,{\mathscr C})\big)=1$, then ${\rm tsr}(\B)=1.$
\end{cor}

The class of Banach $^*$-algebras sitting between $\ell^1( G\,\vert\,{\mathscr C})$ and ${\rm C^*}( G\,\vert\,{\mathscr C})$ seems to be vast. Indeed, for twisted $C^*$-dynamical systems of the form $(G,C_0(X),\alpha,\omega)$, and taking a subset $P\subset [1,\infty]$ as input, Bardadyn and Kwa\'sniewski constructed an `$L^P$-completion' of $\ell^1_{\alpha,\omega}(G,C_0(X))$. If $P$ is symmetric with respect to $2$, this is a Banach $^*$-algebra \cite[Remark 4.13]{BaKw25}, which embeds in $C_0(X)\rtimes_{\alpha,\omega} G$ by the Riesz-Thorin interpolation theorem. 

In particular, applying the above-mentioned theorem to symmetrized twisted $L^p$-crossed products yields the following theorem. In it, we use the notations $F^p_*(G,X,\alpha,\omega)$ and $F^p(G,X,\alpha,\omega)$ for the symmetrized and regular twisted $L^p$-crossed products associated with a twisted $C^*$-dynamical system of the form $(G,C_0(X),\alpha,\omega)$.
\begin{thm}[Theorem \ref{mainsymcor}]\label{mainsym}
    Let $(G,C_0(X),\alpha,\omega)$ be a twisted $C^*$-dynamical system, where $X$ is a Hausdorff locally compact space and $G$ is a group of subexponential growth. Then 
        $$
        {\rm rr}\big(F^q_*(G,X,\alpha,\omega)\big)= {\rm rr}\big(C_0(X)\rtimes_{\alpha,\omega}G\big)\quad \text{ and }\quad {\rm tsr}\big(F^q_*(G,X,\alpha,\omega)\big)={\rm tsr}\big(C_0(X)\rtimes_{\alpha,\omega}G\big),
        $$ 
        for all $q\in[1,2]$. Furthermore, if ${\rm tsr}\big(C_0(X)\rtimes_{\alpha,\omega}G\big)=1$, then one also has
        $$
        {\rm tsr}\big(F^{p}(G,X,\alpha,\omega)\big)=1,\quad\text{ for all }p\in[1,\infty].
        $$
\end{thm}

In particular, we have the following result. It fully settles Question \ref{theq}.
\begin{cor}
    Let $\alpha$ be a minimal action of $\mathbb Z$ either on the Cantor set $X=2^\mathbb N$ or on $X=\mathbb T$ by irrational rotations. Then $F^p(\mathbb Z,X,\alpha)$ has topological stable rank $1$, for all $p\in[1,\infty]$.
\end{cor}

In fact, thanks to the generality of our setting, many other topological ranks are computed in Theorem \ref{list}. They might be regarded as $L^p$-versions of results in \cite{Pu90b,Pu90,BlKuRo92,Ka93,CP05}. It was recently announced by Li and Niu that every crossed product $C(X)\rtimes \Z^d$, coming from a free and minimal action $\Z^d\curvearrowright X$, has topological stable rank $1$ \cite{LiNi23}. This widens the class of algebras for which our theorems are applicable and provides other examples not explicitly mentioned here.

We remark that the results mentioned in the previous paragraph involve many different approaches, some of which are very intricate, while others contain very ad-hoc techniques. Our approach is very general and manages to bypass these methods in a unified manner.

The reader should note that groups of subexponential growth are necessarily amenable; hence, we do not know what the situation is for non-amenable groups. More concretely, we do not know the topological stable rank or the real rank of $F^p_*(G)$ for non-amenable $G$. This question seems particularly interesting for free groups $G=\mathbb F_n$.

The article is organized as follows. Section \ref{prem} deals with preliminaries but also includes the proof of Theorem \ref{mainthm}. It first introduces Fell bundles and their $\ell^1$-algebras and recalls some of the tools that will be later involved in the proofs of the main results. This section is divided into two subsections. The first subsection introduces and studies the topological stable rank for these Fell bundle algebras, while the second subsection does the same for the real rank. Section \ref{main} then briefly discusses the construction of $L^p$-crossed products and symmetrized $L^p$-crossed products and explains how the previously obtained results apply to these algebras. The main theorem of this section is Theorem \ref{list}, which basically consists of a list of results that follow from Theorem \ref{mainsym} and their (already known) appropriate $C^*$-counterparts.

\section{\texorpdfstring{$\ell^1$}--algebras of Fell bundles}\label{prem}

From now on, $G$ will denote a (discrete) amenable group with unit $\e$. We will also fix a Fell bundle $\mathscr C\!=\bigsqcup_{x\in G}\mathfrak C_x$ over $G$. That is, the $\mathfrak C_x$'s are a family of Banach spaces equipped with bilinear, associative multiplication maps $\bu:\mathfrak C_x\times \mathfrak C_y\to\mathfrak C_{xy}$ and conjugate-linear, anti-multiplicative involutions $^\bu: \mathfrak C_x\to\mathfrak C_{x^{-1}}$ satisfying natural properties that turn the unit fiber $\mathfrak C_\e$ into a $C^*$-algebra, and each $\mathfrak C_x$ into a Hilbert $\mathfrak C_\e$-bimodule (see \cite{FD88}).

The associated algebra of summable cross-sections $\ell^1( G\,\vert\,\mathscr C)$ is a Banach $^*$-algebra and a completion of the space $C_{\rm c}( G\,\vert\,\mathscr C)$ of cross-sections with finite support under the norm 
$$
\norm{\Phi}_1=\sum_{x\in G}\norm{\Phi(x)}_{\mathfrak C_x}.
$$ 
The universal $C^*$-algebra completion of $\ell^1( G\,\vert\,\mathscr C)$ is denoted by ${\rm C^*}( G\,\vert\,\mathscr C)$, and the associated universal $C^*$-norm (i.e., the norm in ${\rm C^*}( G\,\vert\,\mathscr C)$) is denoted $\norm{\cdot}_*$. We recall that the product on $\ell^1( G\,\vert\,\mathscr C)$ is given by
\begin{equation*}\label{broduct}
\big(\Phi*\Psi\big)(x)=\sum_{y\in G} \Phi(y)\bu \Psi(y^{-1}x)
\end{equation*}
and the involution is given by
\begin{equation*}\label{inwol}
\Phi^*(x)=\Phi(x^{-1})^\bu\,,
\end{equation*}
in terms of the operations $\big(\bu,^\bu\big)$ on the Fell bundle. Note that the adjoint operation $\Phi\mapsto \Phi^*$ is isometric. 

We will also consider the $\ell^\infty$-norm \begin{equation*}
    \norm{\Phi}_{\infty}=\sup_{x\in G}\norm{\Phi(x)}_{\mathfrak C_x},
\end{equation*} 
primarily because $\norm{\Phi}_{\infty}\leq \norm{\Phi}_{*}$, for all $\Phi\in {\rm C^*}( G\,\vert\,\mathscr C)$ \cite[Lemma 3.4]{Fl24}.

One of the main classes of examples of Fell bundles is presented next. It is the class of Fell bundles coming from twisted $C^*$-dynamical systems.
\begin{ex}\label{mainex}
    Let $\A$ be a $C^*$-algebra. A (continuous) twisted action of $G$ on $\A$ is a pair $(\alpha,\omega)$ of maps $\alpha:G\to{\rm Aut}({\A})$ and $\omega:G\times G\to \mathcal{UM}({\A})$, such that $\alpha$ is strongly continuous, $\omega$ is jointly strictly continuous, and \begin{itemize}
        \item[(i)] $\alpha_x(\omega(y,z))\omega(x,yz)=\omega(x,y)\omega(xy,z)$,
        \item[(ii)] $\alpha_x\big(\alpha_y(a)\big)\omega(x,y)=\omega(x,y)\alpha_{xy}(a)$,
        \item[(iii)] $\omega(x,\e)=\omega(\e,y)=1, \alpha_\e={\rm id}_{{\A}}$,
    \end{itemize} for all $x,y,z\in G$ and $a\in\A$. 

The quadruple $(G,\A,\alpha,\omega)$ is called a \emph{twisted $C^*$-dynamical system}. Given such a twisted action, one usually forms the so-called \emph{twisted convolution algebra} $\ell^1_{\alpha,\omega}( G,\A)$, consisting of all summable sequences $\Phi:G\to\A$, endowed with the product 
\begin{equation*}\label{convolution}
    \Phi*\Psi(x)=\sum_{y\in G} \Phi(y)\alpha_y[\Psi(y^{-1}x)]\omega(y,y^{-1}x)
\end{equation*} 
and the involution 
\begin{equation*}\label{involution}
    \Phi^*(x)=\omega(x,x^{-1})^*\alpha_x[\Phi(x^{-1})^*],
\end{equation*} 
making $\ell^1_{\alpha,\omega}(G,\A)$ a Banach $^*$-algebra under the norm $\norm{\Phi}_{1}=\sum_{x\in G}\norm{\Phi(x)}_{\A}$. When the twist is trivial ($\omega\equiv1$), we omit any mention of it and call the resulting algebra $\ell^1_{\alpha}(G,\A)$ the \emph{convolution algebra} associated with the action $\alpha$. In this case, the triple $(G,\A,\alpha)$ is called an (untwisted) \emph{$C^*$-dynamical system}. The associated $C^*$-envelope is usually denoted $\A\rtimes_{\alpha,\omega}G$ or $\A\rtimes_{\alpha}G$ in the untwisted case.

It is well known that these algebras correspond to our algebras $\ell^1(G\,\vert\,\mathscr C)$ for very particular Fell bundles. In fact, these bundles may be easily described as $\mathscr C=\A\times G$, with the quotient map $q(a,x)=x$, constant norms $\norm{\cdot}_{\mathfrak C_x}=\norm{\cdot}_{\A}$, and operations \begin{equation*}
    (a,x)\bu(b,y)=(a\alpha_x(b)\omega(x,y),xy)\quad\textup{and}\quad (a,x)^\bu=(\alpha_{x}^{-1}(a^*)\omega(x^{-1},x)^*,x^{-1}).
\end{equation*} 
\end{ex}

In what follows, $\widetilde \B$ will denote the minimal unitization of $\B$. If $\B$ is already unital, we simply set $\widetilde \B=\B$. We note that the minimal unitization of $\ell^1( G\,\vert\,\mathscr C)$ corresponds to the $\ell^1$-algebra of the `unitization Fell bundle' $\widetilde{\mathscr C}\to G$, which is constructed from $\mathscr C$ simply by replacing $\mathfrak C_\e$ with $\widetilde{\mathfrak C_\e}$ and keeping all the other fibers intact.

In order to prove our main result, we first need to recall some notions of multipliers. In particular, the following is a well-known characterization of amenability that can be found in \cite[Theorem 2.6.8]{BrOz08}. It uses the notion of positive definite functions, so we will recall that first.

\begin{defn}
    Let $G$ be a group. A function $\varphi:G\to\CC$ is said to be positive definite if, for every finite subset $K\subset G$, the matrix
    $$
    [\varphi(x^{-1}y)]_{x,y\in K}\in \mathbb M_{|K|}(\CC)
    $$
    is positive definite.
\end{defn}

\begin{lem}\label{cpmap}
    A group $G$ is amenable if and only if there exists a net $\{\varphi_j\}_{j\in J}$ of finitely supported positive definite functions on $G$ such that $\varphi_j\to 1$ pointwise and $\varphi_j(\e)\leq 1$, for all $j\in J$.
\end{lem}

Furthermore, B\'edos and Conti \cite[Example 4.4]{BeCo25} showed that a positive definite function $\varphi:G\to\CC$ induces a completely positive map  
$$
T_\varphi: {\rm C^*}( G\,\vert\,\mathscr C)\to {\rm C^*}( G\,\vert\,\mathscr C),\quad \text{ given by }\quad T_\varphi(\Phi)(x)=\varphi(x)\Phi(x).
$$
In fact, if $\varphi_j\to 1$ pointwise and $\varphi_j(\e)\leq 1$ (as in the previous lemma), then $\sup_{x\in G}|\varphi(x)|\leq 1$ \cite[Corollary 3.7]{BeCo25}. Moreover, one has $\norm{T_{\varphi_j}\Phi}_*\leq \norm{\Phi}_*$ and $T_{\varphi_j}\Phi\to \Phi$, for all $\Phi\in C_{\rm c}( G\,\vert\,\mathscr C)$ \cite[Theorem 3.11]{BeCo25}.

The other ingredient we need to introduce comes from spectral theory. Hence, as usual, we will denote the spectrum of an element $b\in\B$ by  
$$
{\rm Spec}_{\B}(b)=\{t\in\CC\mid b-t1\textup{ is not invertible in }\widetilde\B\}.
$$ 
We also recall that a group is said to have subexponential growth if, for all finite subsets $K\subset G$, one has
$$
    \lim_{n\to\infty}|K^n|^{1/n}=1.
$$ 
All nilpotent groups have subexponential growth; thus, they serve as examples. In fact, all of our concrete applications make use of nilpotent groups (see Section \ref{main}).

\begin{thm}[\cite{Fl24}]\label{subexp}
    Suppose that $G$ has subexponential growth. Then \begin{equation*}
        {\rm Spec}_{\ell^1(G\,\vert\,\mathscr C)}(\Phi)={\rm Spec}_{{\rm C^*}( G\,\vert\,\mathscr C)}(\Phi),\quad \text{ for all }\Phi\in C_{\rm c}(G\,\vert\,\mathscr C).
\end{equation*} \end{thm}

Before moving to the next section, we remark that all the Banach $^*$-algebras mentioned in this section have a continuous (actually, isometric) involution. In what follows, the algebras that admit a continuous involution will play a central role.

\subsection{Topological stable rank}

Let $\B$ be a unital Banach algebra and $n\geq 1$. We set
\begin{align*}
{\rm Lg}_n(\B) := \left\{ (b_1,\ldots,b_n)\in \B^n \mid \text{the left ideal generated by $b_1,\ldots,b_n$ coincides with $\B$}\right\}.
\end{align*}
The abbreviation `Lg' stands for `left generators'. If $\B$ has an involution, we also use ${\rm Lg}_n(\B)_{\rm sa} := {\rm Lg}_n(\B)\cap \B^n_{\rm sa},$ for the set of self-adjoint left generators. We now give the definition of (left) topological stable rank, as given by Rieffel in \cite[Definition 1.4]{Ri83}.

\begin{defn}
    The \emph{(left) topological stable rank} of a unital Banach algebra $\B$ is the smallest $n\in\Z^+$ with the property that ${\rm Lg}_n(\B)$ is dense in $\B^n$ under the product topology. In such a case, we write ${\rm tsr}(\B)=n$. On the other hand, we say that ${\rm tsr}(\B)=\infty$ if no such $n$ exists.
\end{defn}

Rieffel observed that we get the same notion if we decide to work with right ideals on a Banach $^*$-algebra, at least when the involution is continuous \cite[Proposition 1.6]{Ri83}. All the results we will prove here also hold for the right-sided topological stable rank, either by virtue of this observation or by analogy. In any case, let us now introduce a characterization for the tuples in ${\rm Lg}_n(\B)$ that is easier to work with. It is well known to experts, but we include it here for completeness (and to highlight the fact that it uses $C^*$-algebraic techniques). In fact, it is sometimes taken as the definition of ${\rm Lg}_n(\B)$ (cf. \cite[pgs. 16-17]{Th21}).

\begin{prop}\label{useful}
    Let $\B$ be a unital $C^*$-algebra. Then $(b_1,\ldots,b_n)\in {\rm Lg}_n(\B)$ if and only if $\sum_{i=1}^n b_i^* b_i $ is invertible.
\end{prop}
\begin{proof}
    If $\sum_{i=1}^n b_i^* b_i $ is invertible, then the left ideal generated by $b_1,\ldots,b_n$ contains an invertible element, and so it coincides with $\B$. On the other hand, and for the sake of contradiction, suppose that $b_1,\ldots,b_n$ generates $\B$ as a left ideal, but that $\sum_{i=1}^n b_i^* b_i $ is not invertible.
    
    Since $b_1,\ldots,b_n$ generates $\B$, we can find $a_1,\ldots,a_n\in\B$ such that
    $$
    1=a_1b_1+\ldots +a_nb_n.
    $$
    Now let $\pi:\B\to\BofH$ be a unital faithful $^*$-representation. Then
    $$
    0\in {\rm Spec}_\B\Big(\sum_{i=1}^n b_i^* b_i\Big)={\rm Spec}_{\BofH}\Big(\sum_{i=1}^n \pi(b_i^* b_i)\Big), 
    $$
    so $S=\sum_{i=1}^n \pi(b_i^* b_i)$ is a positive non-invertible operator on $\mathcal H$. Since $S$ is not bounded below, there exists a unit vector $\xi\in\mathcal H$ such that
    $$
    \langle S\xi,\xi\rangle <\frac{1}{M^2n^2}, \quad\text{ where }M=\max_{1\leq i\leq n}\norm{\pi(a_i)}_{\BofH}.
    $$
    But then, for every $1\leq i\leq n$, we have
    \begin{equation}\label{Delta}
        \norm{\pi(b_i)\xi}_{\mathcal H}^2=\langle \pi(b_i^*b_i)\xi,\xi\rangle<\frac{1}{M^2n^2}
    \end{equation}
    and 
    $$
    1=\norm{\xi}_{\mathcal H}=\Big\|\sum_{i=1}^n \pi(a_ib_i)\xi\Big\|_{\mathcal H}\leq \sum_{i=1}^n \|\pi(a_i)\|_{\BofH}\|\pi(b_i)\xi\|_{\mathcal H}\leq M \sum_{i=1}^n \|\pi(b_i)\xi\|_{\mathcal H}\overset{\eqref{Delta}}{<}1,
    $$
    which is indeed a contradiction.
\end{proof}

\begin{lem}\label{major}
    Let $\A\subset \B$ be a dense, continuous unital inclusion of Banach algebras. Then ${\rm tsr}(\B)\leq {\rm tsr}(\A).$
\end{lem}
\begin{proof}
    Let ${\rm tsr}(\A)=n<\infty$, let $\varepsilon>0$, and let $b_i\in\B$, for $1\leq i\leq n$. Because the inclusion of $\A$ in $\B$ is dense, we can find elements $a'_i\in\A$ and $1\leq i\leq n$ such that $\norm{a'_i-b_i}_\B<\varepsilon/2$. Now use the definition of ${\rm tsr}(\A)=n$ and the continuity of $\A\subset\B$ to choose elements $a_i,c_i\in\A$, $1\leq i\leq n$, such that $\norm{a_i-a'_i}_\B<\varepsilon/2$ and 
    $$
    1=\sum_{i=1}^n c_ia_i.
    $$
    But then the left ideal generated by $a_1,\ldots,a_n$ in $\B$ contains $1$, so $(a_1,\ldots,a_n) \in {\rm Lg}_n(\B)$. Since 
    $$
    \norm{a_i-b_i}_\B\leq \norm{a_i-a'_i}_\B+\norm{a'_i-b_i}_\B <\varepsilon,
    $$
    we see that ${\rm Lg}_n(\B)$ is dense in $\B^n$, so ${\rm tsr}(\B)\leq n$.
\end{proof}

We now prove the main theorem of the section. The underlying idea is to use the characterization of ${\rm Lg}_n(\B)$ in terms of invertibility and then appeal to Theorem \ref{subexp} to ensure that the corresponding inverses remain inside $\ell^1( G\,\vert\,{\mathscr C})$. Lemma \ref{cpmap} then provides us with a way of upgrading the convergence in ${\rm C^*}( G\,\vert\,{\mathscr C})$ to convergence in $\ell^1( G\,\vert\,{\mathscr C})$. 

\begin{thm}\label{tsrpreserve}
    Suppose that $G$ has subexponential growth. Let $\B$ be a Banach algebra such that $\ell^1( G\,\vert\,{\mathscr C})\subset\B\subset {\rm C^*}( G\,\vert\,{\mathscr C})$. Then ${\rm tsr}(\B)= {\rm tsr}\big(\ell^1( G\,\vert\,{\mathscr C})\big).$
\end{thm}
\begin{proof}
    It is enough to prove the unital case, as passing to the unitization still gives a Fell bundle algebra that satisfies $\ell^1( G\,\vert\,\widetilde{\mathscr C})\subset\widetilde\B\subset {\rm C^*}( G\,\vert\,\widetilde{\mathscr C})$. Applying Lemma \ref{major}, we have 
    $$
    {\rm tsr}\big({\rm C^*}( G\,\vert\,{\mathscr C})\big)\leq {\rm tsr}(\B)\leq {\rm tsr}\big(\ell^1( G\,\vert\,{\mathscr C})\big).
    $$
    So we only need to prove that ${\rm tsr}\big({\rm C^*}( G\,\vert\,{\mathscr C})\big)\geq {\rm tsr}\big(\ell^1( G\,\vert\,{\mathscr C})\big)$.
    
     Indeed, assume that ${\rm tsr}\big({\rm C^*}( G\,\vert\,{\mathscr C})\big)=n<\infty$, let $\varepsilon>0$, and let $\Phi_1,\ldots,\Phi_n\in C_{\rm c}( G\,\vert\,\mathscr C)$ with $M=\max_{1\leq i\leq n}\norm{\Phi_i}_{1}$. Appealing to Lemma \ref{cpmap} and the subsequent discussion regarding the net $(T_{{\varphi_j}})_j$, we can find a finitely supported positive definite function $\varphi:G\to\CC$ such that 
    \begin{equation}\label{ineq1}
        \norm{\Phi_i-T_\varphi\Phi_i}_{1}<\frac{\varepsilon}{2},\quad\text{ for all }1\leq i\leq n.
    \end{equation}
    Now, we use the definition of topological stable rank to find $\Psi_1,\ldots, \Psi_n\in{\rm C^*}( G\,\vert\,{\mathscr C})$ that generates ${\rm C^*}( G\,\vert\,{\mathscr C})$ as a left ideal such that
    \begin{equation}\label{ineq2}
        \norm{\Psi_i-\Phi_i}_{*}<\frac{\varepsilon}{4|{\rm Supp}(\varphi)|},\quad\text{ for all }1\leq i\leq n.
    \end{equation}
    Because of Proposition \ref{useful}, $\sum_{i=1}^n \Psi_i^**\Psi_i$ is invertible in ${\rm C^*}( G\,\vert\,\mathscr C)$. We now show that $T_\varphi \Psi_i$ approximates $\Phi_i$ in the $\ell^1$-norm.
    \begin{align*}
        \norm{T_\varphi\Psi_i-\Phi_i}_{1}&\leq \norm{T_\varphi\Psi_i-T_\varphi\Phi_i}_{1}+\norm{\Phi_i-T_\varphi\Phi_i}_{1} \\
        &\overset{\eqref{ineq1}}{\leq} \norm{\Psi_i-\Phi_i}_{\infty}|{\rm Supp}(\varphi)|+\frac{\varepsilon}{2} \\
        &\leq \norm{\Psi_i-\Phi_i}_{*}|{\rm Supp}(\varphi)|+\frac{\varepsilon}{2} \\
        &\overset{\eqref{ineq2}}{<}\frac{3}{4}\varepsilon.
    \end{align*}
    In particular, using the inequality in \eqref{ineq2} again, we see that 
    \begin{equation}\label{inter}
        \|T_\varphi\Psi_i\|_1\leq \|\Phi_i\|_{1}+\frac{3}{4}\varepsilon\quad \text{ and }\quad \|\Psi_i\|_*\leq \|\Phi_i\|_{*}+\frac{\varepsilon}{4}.
    \end{equation}
    We also note that 
    \begin{align*}
        \Big\|\sum_{i=1}^n \Psi_i^**\Psi_i-\sum_{i=1}^n (T_\varphi\Psi_i)^**T_\varphi\Psi_i\Big\|_{*}&\leq \sum_{i=1}^n\norm{\Psi_i^**\Psi_i- (T_\varphi\Psi_i)^**T_\varphi\Psi_i}_{*} \\
        &\leq \sum_{i=1}^n\norm{\Psi_i-T_\varphi\Psi_i}_{*}(\norm{\Psi_i}_{*}+\norm{T_\varphi\Psi_i}_{*}) \\
        &\overset{\eqref{inter}}{\leq} \sum_{i=1}^n\norm{\Psi_i-T_\varphi\Psi_i}_{*}(\norm{\Phi_i}_{1}+\norm{T_\varphi\Psi_i}_{*}+\frac{\varepsilon}{4}) \\
        &\overset{\eqref{inter}}{\leq} (2\max_{1\leq i\leq n}\norm{\Phi_i}_1+\varepsilon)\sum_{i=1}^n\norm{\Psi_i-T_\varphi\Psi_i}_{*}, 
    \end{align*}
    where the norms inside the sum can be estimated as
    \begin{align*}
        \norm{\Psi_i-T_\varphi\Psi_i}_{*}&\leq \norm{\Psi_i-\Phi_i}_{*} +\norm{\Phi_i-T_\varphi\Phi_i}_{*}+\norm{T_\varphi\Phi_i-T_\varphi\Psi_i}_{*} \\
        &\leq 2\norm{\Psi_i-\Phi_i}_{*} +\norm{\Phi_i-T_\varphi\Phi_i}_{1} \\
        &\overset{\eqref{ineq1}+\eqref{ineq2}}{<} \varepsilon.
    \end{align*}
    Hence, returning to the main line of computation, we see that
     \begin{align*}
         \Big\|\sum_{i=1}^n \Psi_i^**\Psi_i-\sum_{i=1}^n (T_\varphi\Psi_i)^**T_\varphi\Psi_i\Big\|_{*}\leq n\varepsilon(2M+\varepsilon).
     \end{align*}
    So, if $\varepsilon$ is small enough, $\sum_{i=1}^n (T_\varphi\Psi_i)^**T_\varphi\Psi_i$ is close to $\sum_{i=1}^n \Psi_i^**\Psi_i$, so it has to be invertible. But $\sum_{i=1}^n (T_\varphi\Psi_i)^**T_\varphi\Psi_i\in C_{\rm c}( G\,\vert\,\mathscr C)$ and, by applying Theorem \ref{subexp}, we see that it is invertible in $\ell^1( G\,\vert\,\mathscr C)$. Reasoning as before, we conclude that $T_\varphi\Psi_1,\ldots, T_\varphi\Psi_n$ generates $\ell^1( G\,\vert\,\mathscr C)$ as a left ideal. Hence, $(\Phi_1,\ldots,\Phi_n)$ belongs to the closure of ${\rm Lg}_n\big(\ell^1( G\,\vert\,\mathscr C)\big)$, and this proves that ${\rm tsr}\big(\ell^1( G\,\vert\,\mathscr C)\big)\leq n$.
\end{proof}

\begin{cor}
    Suppose that $G$ has subexponential growth. Then ${\rm tsr}(\ell^1( G\,\vert\,{\mathscr C})\big)= {\rm tsr}\big({\rm C^*}( G\,\vert\,{\mathscr C})\big).$
\end{cor}

Simply specializing to the twisted $C^*$-dynamical system case (i.e. the setting of Example \ref{mainex}) allows us to write the following corollary.

\begin{cor}
    Suppose that $G$ has subexponential growth and let $(G,\A,\alpha,\omega)$ be a twisted $C^*$-dynamical system. Then ${\rm tsr}(\ell^1_{\alpha,\omega}(G,\A)\big)= {\rm tsr}\big(\A\rtimes_{\alpha,\omega}G\big).$
\end{cor}

Finally, in the case that $\B$ lacks an involution and is only assumed to be a Banach algebra, we can achieve the desired equality of stable ranks only in the case of topological stable rank 1.

\begin{cor}\label{easycor}
    Suppose that $G$ has subexponential growth and let $\ell^1( G\,\vert\,{\mathscr C})\subset\B$ be any continuous, dense inclusion into a Banach algebra $\B$. Then 
    $$
    {\rm tsr}(\B)\leq {\rm tsr}\big({\rm C^*}( G\,\vert\,{\mathscr C})\big).
    $$
    In particular, if ${\rm tsr}\big({\rm C^*}( G\,\vert\,{\mathscr C})\big)=1$, then ${\rm tsr}(\B)=1.$
\end{cor}
\begin{proof}
    Combining Theorem \ref{tsrpreserve} with Lemma \ref{major} yields 
    $$
    {\rm tsr}(\B)\leq {\rm tsr}\big(\ell^1( G\,\vert\,{\mathscr C})\big)={\rm tsr}\big({\rm C^*}( G\,\vert\,{\mathscr C})\big),
    $$ 
    which proves the claim.
\end{proof}

\subsection{Real rank}

We now proceed to our study of the real rank. The reader will notice that the similarity of definitions allows our line of argumentation to remain essentially the same, provided that we take care of self-adjointness and the corresponding shift in indices.

\begin{defn}
    The \emph{real rank} of a unital Banach $^*$-algebra $\B$ is the smallest $n\in\N$ with the property that ${\rm Lg}_{n+1}(\B)_{\rm sa}$ is dense in $\B^{n+1}_{\rm sa}$ under the product topology. In such a case, we write ${\rm rr}(\B)=n$. On the other hand, we say that ${\rm rr}(\B)=\infty$ if no such $n$ exists.
\end{defn}

As in the case of the topological stable rank, we can provide a more manageable characterization, at least for $C^*$-algebras.

\begin{prop}\label{useful1}
    Let $\B$ be a unital $C^*$-algebra. Then $(b_1,\ldots,b_n)\in {\rm Lg}_n(\B)_{\rm sa}$ if and only if $\sum_{i=1}^n b_i^2 $ is invertible and $b_i=b_i^*$ for all $1\leq i\leq n$.
\end{prop}
\begin{proof}
    Immediate application of Proposition \ref{useful}.
\end{proof}

\begin{lem}\label{major1}
    Let $\A\subset \B$ be a dense, continuous unital inclusion of Banach $^*$-algebras. If $\A$ has a continuous involution, then ${\rm rr}(\B)\leq {\rm rr}(\A).$
\end{lem}
\begin{proof}
    The proof is the same as the proof of Lemma \ref{major}, but we include it here. Let ${\rm rr}(\A)=n<\infty$, let $\varepsilon>0$, and let $b_i\in\B_{\rm sa}$ for $0\leq i\leq n$. 

    Because the inclusion of $\A$ in $\B$ is dense, we can find self-adjoint elements $a'_i\in\A$, $0\leq i\leq n$, such that $\norm{a'_i-b_i}_\B<\varepsilon/2$ (if needed, replace $a'_i$ with $\tfrac{1}{2}a'_i+\tfrac{1}{2}(a'_i)^*$). Now, we proceed as before and use the definition of ${\rm rr}(\A)=n$ and the continuity of $\A\subset\B$ to choose elements $a_i,c_i\in\A_{\rm sa}$, $0\leq i\leq n$, such that $\norm{a_i-a'_i}_\B<\varepsilon/2$ and 
    $$
    1=\sum_{i=0}^n c_ia_i.
    $$
    But then the left ideal generated by $a_0,\ldots,a_n$ coincides with $\B$, so $(a_0,\ldots,a_n) \in {\rm Lg}_{n+1}(\B)_{\rm sa}$. Since 
    $$
    \norm{a_i-b_i}_\B\leq \norm{a_i-a'_i}_\B+\norm{a'_i-b_i}_\B <\varepsilon,
    $$
    we get that ${\rm Lg}_{n+1}(\B)_{\rm sa}$ is dense in $\B^{n+1}$.
\end{proof}

We now prove the real rank analog of Theorem \ref{tsrpreserve}. The proof follows the same lines.

\begin{thm}\label{rrpreserve}
    Suppose that $G$ has subexponential growth. Let $\B$ be a Banach $^*$-algebra with a continuous involution and such that $\ell^1( G\,\vert\,{\mathscr C})\subset\B\subset {\rm C^*}( G\,\vert\,{\mathscr C})$. Then ${\rm rr}(\B)= {\rm rr}\big(\ell^1( G\,\vert\,{\mathscr C})\big).$
\end{thm}
\begin{proof}
    As before, it is enough to prove the unital case, and because of Lemma \ref{major1}, we only need to prove that ${\rm rr}\big({\rm C^*}( G\,\vert\,{\mathscr C})\big)\geq {\rm rr}\big(\ell^1( G\,\vert\,{\mathscr C})\big)$.
    
    Suppose that ${\rm rr}\big({\rm C^*}( G\,\vert\,{\mathscr C})\big)=n<\infty$. We let $\varepsilon>0$ and $\Phi_0,\ldots,\Phi_n\in C_{\rm c}( G\,\vert\,\mathscr C)_{\rm sa}$ with $M=\max_{0\leq i\leq n}\norm{\Phi_i}_{1}$. Appealing to Lemma \ref{cpmap} and the subsequent discussion regarding the net $(T_{{\varphi_j}})_j$, we can find a finitely supported positive definite function $\varphi:G\to\CC$ such that 
    \begin{equation*}\label{ineq1r}
        \norm{\Phi_i-T_\varphi\Phi_i}_{1}<\frac{\varepsilon}{2},\quad\text{ for all }0\leq i\leq n.
    \end{equation*}
    Now we use the definition of real rank to find $\Psi_0,\ldots, \Psi_n\in {\rm C^*}( G\,\vert\,{\mathscr C})_{\rm sa}$ that generate ${\rm C^*}( G\,\vert\,{\mathscr C})$ as a left ideal and such that
    \begin{equation*}\label{ineq2r}
        \norm{\Psi_i-\Phi_i}_{*}<\frac{\varepsilon}{4|{\rm Supp}(\varphi)|},\quad\text{ for all }0\leq i\leq n.
    \end{equation*}
    By Lemma \ref{useful1}, $\sum_{i=0}^n \Psi_i^2$ is invertible in ${\rm C^*}( G\,\vert\,\mathscr C)$. Furthermore, by the same arguments as in the proof of Theorem \ref{tsrpreserve}, we have
    \begin{align*}
        \norm{T_\varphi\Psi_i-\Phi_i}_{1}<\varepsilon \quad\text{ and }\quad\Big\|\sum_{i=0}^n \Psi_i^2-\sum_{i=0}^n (T_\varphi\Psi_i)^2\Big\|_{*}\leq \varepsilon(n+1)(2M+\varepsilon).
    \end{align*}
For small $\varepsilon$, $\sum_{i=0}^n (T_\varphi\Psi_i)^2$ will be close to $\sum_{i=0}^n \Psi_i^2$, so it will need to be invertible. But $\sum_{i=0}^n (T_\varphi\Psi_i)^2\in C_{\rm c}( G\,\vert\,\mathscr C)$, again by Theorem \ref{subexp}, the inverse lies in $\ell^1( G\,\vert\,\mathscr C)$. Because $\varphi$ is positive definite, this implies $(T_\varphi\Psi_0,\ldots, T_\varphi\Psi_n)\in {\rm Lg}_{n+1}\big(\ell^1( G\,\vert\,\mathscr C)\big)_{\rm sa}$, and thus ${\rm rr}\big(\ell^1( G\,\vert\,\mathscr C)\big)\leq n$.
\end{proof}

\begin{cor}
    Suppose that $G$ has subexponential growth. Then ${\rm rr}(\ell^1( G\,\vert\,{\mathscr C})\big)= {\rm rr}\big({\rm C^*}( G\,\vert\,{\mathscr C})\big).$
\end{cor}

To finish the section, we specialize the last corollary to the twisted $C^*$-dynamical system setting.

\begin{cor}
    Suppose that $G$ has subexponential growth and let $(G,\A,\alpha,\omega)$ be a twisted $C^*$-dynamical system. Then ${\rm rr}(\ell^1_{\alpha,\omega}(G,\A)\big)= {\rm rr}\big(\A\rtimes_{\alpha,\omega}G\big).$
\end{cor}

\section{Applications to symmetrized twisted \texorpdfstring{$L^p$}--crossed products}\label{main}

As mentioned before, the purpose of this section is to apply the previous results to the case of symmetrized $L^p$-crossed products. For that purpose, we will fix a twisted $C^*$-dynamical system $(G,C_0(X),\alpha,\omega)$, where $G$ is a (discrete) group of subexponential growth, $X$ a Hausdorff locally compact space, $\alpha$ an action of $G$ on $X$, and $\omega:G\times G\to \mathcal U(C_b(X))$ a $2$-cocycle. Here, we abuse the notation and call $\alpha$ both the action on $X$ and the induced action on $C_0(X)$. That is
$$
\alpha_x(f)(z)=f\big(\alpha_{x^{-1}}(z)),
$$
where $x\in G,z\in X, f\in C_0(X)$.

It is well known that if $\B$ has a contractive approximate unit and $\pi:\B\to \mathbb B(\E)$ is a non-degenerate, contractive representation on a Banach space $\E$, then $\pi$ extends uniquely to a (unital) representation $\pi:\mathcal M(\B)\to \mathbb B(\E)$ of the multiplier algebra of $\B$. This is always the case for $\B=C_0(X)$, and it is relevant since our twisting $2$-cocycles are $\mathcal M(C_0(X))=C_b(X)$-valued. Note that the representation $\pi:\mathcal M(\B)\to \mathbb B(\E)$ remains contractive \cite[Theorem 3.4]{BiDeWe24}.

\begin{defn}\label{D-CovRep}
Let $p \in [1, \infty],$ and let $(G,C_0(X),\alpha,\omega)$ be a twisted $C^*$-dynamical system as described above. Let $(Y,\mathcal B,\mu)$ be a measure space. A \emph{covariant representation} of $(G,C_0(X),\alpha,\omega)$ on $L^p (Y, \mu)$ is a pair $(v, \pi)$ consisting of a map $g \mapsto v_g$ from $G$ to the invertible isometries on $L^p (Y, \mu)$ and a non-degenerate, contractive representation $\pi : C_0(X) \to \mathbb B(L^p (Y, \mu)),$ such that
$$
v_xv_y={\pi}(\omega(x,y))v_{xy}, \qquad  v_x\pi(f)v_{x}^{-1}=\pi(\alpha_x(f)),
$$
for all $x,y \in G$ and $f \in C_0(X)$.

Given a covariant representation $(v, \pi)$ on $L^p (Y, \mu)$, we define the associated \emph{integrated
representation} $\pi\rtimes v : \ell^1_{\alpha,\omega}(G,C_0(X))\to \mathbb B(L^p (Y, \mu))$ by
$$
(\pi\rtimes v)(\Phi)=\sum_{x\in G} \pi(\Phi(x))v_x,  
$$
for all $\Phi\in \ell^1_{\alpha,\omega}(G,C_0(X))$.
\end{defn}

With these definitions at hand, we can introduce the following norm on $\ell^1_{\alpha,\omega}(G,C_0(X))$:
$$
        \|\Phi\|_{p,{\rm max}}=\sup\{\|(\pi\rtimes v)(\Phi)\|\mid (v,\pi) \text{ is a covariant representation on some space }L^p (Y, \mu)\}.
$$
The completion $F^p(G,X,\alpha,\omega):=\overline{\ell^1_{\alpha,\omega}(G,C_0(X))}^{\|\cdot\|_{p,{\rm max}}}$ is usually called the full $L^p$-crossed product associated with $(G,C_0(X),\alpha,\omega)$. It is easily seen to be a Banach algebra\footnote{The authors in \cite{BaKw25,DeFaPa25} also consider other, more general, completions of the same kind. That is, completing $\ell^1_{\alpha,\omega}(G,C_0(X))$ with respect to certain classes of representations. However, their general definitions only yield seminorms instead of norms. Here, the special feature of the $L^p$-case is the existence of regular representations, which are injective.}. In the case where the action is trivial, we set $F^p(G,\omega):=F^p(G,\{\rm pt\},1,\omega)$. If the cocycle is also trivial, then we set $F^p(G):=F^p(G,\{\rm pt\},1,1)$. These algebras are often called the twisted group $L^p$-operator algebra and the group $L^p$-operator algebra, respectively. 

These algebras have appeared in full generality in \cite{BaKw25} and in \cite{DeFaPa25}. Note, however, that the untwisted crossed product also appears in \cite{Ga21,ChGaTh24}. Twisted group algebras appear in \cite{HeOr23}, while \cite{El25} deals exclusively with group algebras. For the groupoid case, we recommend \cite{AuOr22}.

It is observed in \cite[Lemma 2.11]{BaKw25} that the formulas 
$$
\pi(f)\xi(x,t)=f(x)\xi(x,t), \quad v_y \xi(x,t)=\omega(y,y^{-1}t)(x)\xi(\alpha_{y^{-1}}(x),y^{-1}t),
$$
valid for $f\in C_0(X)$, $x,y,t\in G$, and $\xi\in \ell^p(X\times G)$, define a covariant representation of $\ell^1_{\alpha,\omega}(G,C_0(X))$ on $\ell^p(X\times G)$. Its integrated form satisfies the formula (note the special notation)
$$
\Lambda_p(\Phi)\xi(x,t)=\sum_{y\in G} \Phi(y)(x)\omega(y,y^{-1}t)(x)\xi(\alpha_{y^{-1}}(x),y^{-1}t).
$$
Using this representation, one may define the reduced twisted crossed product as the completion $\overline{\Lambda_p\big(\ell^1_{\alpha,\omega}(G,C_0(X))\big)}\subset \mathbb B(\ell^p(X\times G))$. However, in our setting (the groups are amenable), this construction is known to coincide with the full crossed product \cite[Theorem 3.5]{BaKw25}. 

In any case, the representation $\Lambda_p$ allows us to introduce a `symmetrized' version of the $L^p$-crossed product. This construction appears explicitly in \cite{BaKw25}; similar constructions also appear in \cite{AuOr22,el24b,El25,DeFaPa25}, with the only difference being the degree of generality that these authors allow. In any case, consider the norm 
$$
\norm{\Phi}_{q,*}=\max\{\|\Lambda_p(\Phi)\|,\|\Lambda_q(\Phi)\|\},
$$
where $p,q\in [1,\infty]$ are H\"older duals. We define $F^q_*(G,X,\alpha,\omega)$ to be the completion of $\ell^1_{\alpha,\omega}(G,C_0(X))$ under the norm $\norm{\cdot}_{q,*}$. As before, we define the symmetrized group algebra and the symmetrized twisted group algebra by setting $F^q_*(G,\omega):=F^q_*(G,\{\rm pt\},1,\omega)$ and $F^q_*(G):=F^q_*(G,\{\rm pt\},1,1)$.

The point is that $F^q_*(G,X,\alpha,\omega)$ is naturally a Banach $^*$-algebra that sits between $\ell^1_{\alpha,\omega}(G,C_0(X))$ and $C_0(X)\rtimes_{\alpha,\omega}G$, so it fits our setting perfectly. Indeed, the fact that this algebra has a continuous involution arises from the fact that, for all $\Phi\in \ell^1_{\alpha,\omega}(G,C_0(X)),\xi\in \ell^p(X\times G),\eta\in \ell^q(X\times G)$, one has
$$
\langle \Lambda_p(\Phi)\xi,\eta\rangle=\langle\xi,\Lambda_q(\Phi^*)\eta\rangle,
$$
where $\langle\cdot,\cdot\rangle$ denotes the duality pairing between $\ell^p(X\times G)$ and $\ell^q(X\times G)$ (cf. \cite{AuOr22}). On the other hand, the inclusions $\ell^1_{\alpha,\omega}(G,C_0(X))\subset F^q_*(G,X,\alpha,\omega)\subset C_0(X)\rtimes_{\alpha,\omega}G$ follow from complex interpolation. Indeed, note that $F^q_*(G,X,\alpha,\omega)$ acts on both $\ell^p(X\times G)$ and $\ell^q(X\times G)$, so an immediate application of the Riesz-Thorin interpolation theorem yields the existence of $\theta\in[0,1]$ such that
$$
\norm{\Lambda_2(\Phi)}\leq \norm{\Lambda_p(\Phi)}^{1-\theta}\norm{\Lambda_q(\Phi)}^\theta\leq \norm{\Phi}_{q,*}.
$$

We remit the reader to \cite[Section 3]{AuOr22} or \cite[Sections 2,3]{BaKw25} for a careful discussion of the algebra $F^q_*(G,X,\alpha,\omega)$ and the properties that we have just mentioned. The interpolation argument can be explicitly found in \cite{AuOr22} or in \cite{El25}. In \cite{BaKw25}, a different argument is presented. Note that \cite[Definition 4.12]{BaKw25} provides a more general construction that takes a set $P\subset[1,\infty]$ as input. In \cite[Remark 4.13]{BaKw25}, the authors provide a sufficient condition for this algebra to be a Banach $^*$-algebra.

In any case, Theorem \ref{tsrpreserve} and Theorem \ref{rrpreserve} apply to $F^q_*(G,X,\alpha,\omega)$, while Corollary \ref{easycor} applies to $F^{p}(G,X,\alpha,\omega)$. The following theorem is immediately derived as a consequence. 

\begin{thm}\label{mainsymcor}
    Let $(G,C_0(X),\alpha,\omega)$ be a twisted $C^*$-dynamical system, where $X$ is a Hausdorff locally compact space and $G$ is a group of subexponential growth. Then 
        $$
        {\rm rr}\big(F^q_*(G,X,\alpha,\omega)\big)= {\rm rr}\big(C_0(X)\rtimes_{\alpha,\omega}G\big)\quad \text{ and }\quad {\rm tsr}\big(F^q_*(G,X,\alpha,\omega)\big)={\rm tsr}\big(C_0(X)\rtimes_{\alpha,\omega}G\big),
        $$ 
        for all $q\in[1,2]$. Furthermore, if ${\rm tsr}\big(C_0(X)\rtimes_{\alpha,\omega}G\big)=1$, then one also has
        $$
        {\rm tsr}\big(F^{p}(G,X,\alpha,\omega)\big)=1,\quad\text{ for all }p\in[1,\infty].
        $$
\end{thm}

We finish the article with the following theorem, which compiles some consequences and explicit computations derived from these results.

\begin{thm}\label{list}
    The following statements are true.
    \begin{itemize}
        \item[(i)] Let $G$ be an abelian group, whose Pontryagin dual is denoted $\widehat G$. Then, for $p\in[1,2]$,
        $$
        {\rm rr}\big(F^p_*(G)\big)={\rm dim}\big(\widehat G\big)\quad\text{ and }\quad {\rm tsr}\big(F^p_*(G)\big)=\floor{\tfrac{1}{2}{\rm dim}\big(\widehat G\big)}+1, 
        $$
        where ${\rm dim}(X)$ is the covering dimension of a compact space $X$.
        \item[(ii)] Let $G$ be a nilpotent group, and let $G^f\subset G$ denote the (normal) subgroup of torsion elements. Then, for $p\in[1,2]$, \begin{enumerate}
            \item ${\rm rr}\big(F^p_*(G)\big)=0$ if and only if $G=G^f$.
            \item ${\rm rr}\big(F^p_*(G)\big)=1$ if and only if $G/G^f\cong \Z$.
        \end{enumerate}
        \item[(iii)] Let $\alpha$ be a free, minimal action of $\Z^d$ on the Cantor set $X=2^\N$. Then 
        $$
        {\rm rr}\big(F^q_*(\Z^d,X,\alpha)\big)=0\quad \text{ and }\quad {\rm tsr}\big(F^q_*(\Z^d,X,\alpha)\big)={\rm tsr}\big(F^{p}(\Z^d,X,\alpha)\big)=1,
        $$ 
        for all $p\in[1,\infty)$ and $q\in[1,2]$.
        \item[(iv)] Let $\alpha$ be an action of $\Z$ on $\T$ by irrational rotations. Then
        $$
        {\rm rr}\big(F^q_*(\Z,\T,\alpha)\big)=0\quad \text{ and }\quad {\rm tsr}\big(F^q_*(\Z,\T,\alpha)\big)={\rm tsr}\big(F^{p}(\Z,\T,\alpha)\big)=1,
        $$ 
        for all $p\in[1,\infty)$ and $q\in[1,2]$.
        \item[(v)] Let $G$ be a group of subexponential growth and let $\alpha$ be the canonical action of $G$ on the Stone-\v{C}ech compactification $\beta G$. Then $G$ is locally finite if and only if ${\rm tsr}(F^q_*(G,\beta G,\alpha))=1$ for some (and hence any) $q\in[1,2]$.
    \end{itemize}
\end{thm}
\begin{proof}
    All the affirmations follow from Theorem \ref{mainsymcor}, Corollary \ref{easycor}, and the observation that the appropriate result also holds in the corresponding $C^*$-case. Indeed, \emph{(i)} follows from \cite{Ri83,BP91}; \emph{(ii)} follows from \cite[Theorem 2]{Ka93}; \emph{(iii)} follows from \cite{CP05}; \emph{(iv)} follows from \cite{Pu90,BlKuRo92}; and \emph{(v)} follows from \cite[Corollary 1.3]{LiLi18}.
\end{proof}

\section*{Acknowledgments}

The author gratefully acknowledges support by the NSF grant DMS-2144739. He is grateful to Professors Ben Hayes, Hannes Thiel and the anonymous referee for their helpful comments. The author also thanks Iason Moutzouris, who pointed him to Proposition \ref{useful}, and Professors Jianguo Zhang and Bartosz Kwaśniewski, for pointing him to the references \cite{LiNi23} and \cite{BaKw25}, respectively.

\printbibliography

\bigskip
\bigskip
ADDRESS

\smallskip
Felipe I. Flores

Department of Mathematics, University of Virginia,

114 Kerchof Hall. 141 Cabell Dr,

Charlottesville, Virginia, United States

E-mail: hmy3tf@virginia.edu

\end{document}